\documentclass[a4paper,11pt,reqno]{amsart}

\usepackage[margin=1in,includehead,includefoot,asymmetric,marginparwidth=1cm]{geometry}
\usepackage[utf8]{inputenc}
\usepackage[T1]{fontenc}
\usepackage{lmodern,microtype}
\usepackage{hyperref}
\usepackage{url}
\usepackage{cancel}
\usepackage{booktabs}
\usepackage{mathtools,amssymb,amsthm}
\usepackage[foot]{amsaddr}
\usepackage{graphicx,color}
\usepackage[margin=5mm]{caption}
\usepackage{enumitem}
\usepackage[b]{esvect} 
\usepackage[textsize=footnotesize]{todonotes}
\usepackage{tikz}
\usepackage{mycommands}
\graphicspath{{./graphics/}}
\usetikzlibrary {arrows.meta}

\newtheorem{theorem}{Theorem}

\newtheorem{lemma}[theorem]{Lemma}

\newtheorem{corollary}[theorem]{Corollary}


\hypersetup{
  pdftitle={On minimum Venn diagrams},
  pdfauthor={Sofia Brenner, Petr Gregor, Torsten M\"utze, Francesco Verciani}
}

\title{On minimum Venn diagrams}

\author{Sofia Brenner}
\address[Sofia Brenner]{Institut f\"ur Mathematik, Universit\"at Kassel, Germany}
\email{sbrenner@mathematik.uni-kassel.de}

\author{Petr Gregor}
\address[Petr Gregor]{Department of Theoretical Computer Science and Mathematical Logic, Charles University, Prague, Czech Republic}
\email{gregor@ktiml.mff.cuni.cz}

\author{Torsten M\"utze}
\address[Torsten M\"utze]{Institut f\"ur Mathematik, Universit\"at Kassel, Germany}
\email{tmuetze@mathematik.uni-kassel.de}

\author{Francesco Verciani}
\address[Francesco Verciani]{Institut f\"ur Mathematik, Universit\"at Kassel, Germany}
\email{francesco.verciani@mathematik.uni-kassel.de}

\thanks{Sofia Brenner, Torsten M\"utze and Francesco Verciani were supported by German Science Foundation grant~522790373.}

\begin{document}

\begin{abstract}
An \emph{$n$-Venn diagram} is a diagram in the plane consisting of $n$ simple closed curves that intersect only finitely many times such that  each of the $2^n$ possible intersections is represented by a single connected region.
An $n$-Venn diagram has at most $2^n-2$ crossings, and if this maximum number of crossings is attained, then only two curves intersect in every crossing.
To complement this, Bultena and Ruskey considered $n$-Venn diagrams that minimize the number of crossings, which implies that many curves intersect in every crossing.
Specifically, they proved that the total number of crossings in any $n$-Venn diagram is at least $L_n\ass\lceil\frac{2^n-2}{n-1}\rceil$, and if this lower bound is attained then essentially all $n$ curves intersect in every crossing.
Diagrams achieving this bound are called \emph{minimum} Venn diagrams, and are known only for $n\leq 7$.
Bultena and Ruskey conjectured that they exist for all $n\geq 8$.
In this work, we establish an asympototic version of their conjecture.
For $n=8$ we construct a diagram with 40 crossings, only 3 more than the lower bound~$L_8=37$.
Furthermore, for every $n$ of the form $n=2^k$ for some integer $k\geq 4$, we construct an $n$-Venn diagram with at most $(1+\frac{33}{8n})L_n=(1+o(1))L_n$ many crossings.
Via a doubling trick this also gives $(n+m)$-Venn diagrams for all $0\leq m<n$ with at most $40\cdot 2^m$ crossings for $n=8$ and at most $(1+\frac{33}{8n})\frac{n+m}{n}L_{n+m}=(2+o(1))L_{n+m}$ many crossings for $k\geq 4$.
In particular, we obtain $n$-Venn diagrams with the smallest known number of crossings for all $n\geq 8$.
Our constructions are based on partitions of the hypercube into isometric paths and cycles, using a result of Ramras.
\end{abstract}

\maketitle

\section{Introduction}
\label{sec:intro}

An \defi{$n$-Venn diagram} is a collection of $n\geq 1$ simple closed curves in the plane that intersect in only finitely many points and that create exactly $2^n$ connected regions, one for every possible combination of being inside or outside with respect to each curve.
It is easy to see that in any intersection point of at least two curves, at least two of them must cross, and thus we refer to any such intersection point as a \defi{crossing}.
A Venn diagram is \defi{simple} if every crossing involves only two curves; see Figure~\ref{fig:dia1}.

Venn diagrams are an appealing tool to visualize sets and their containment relations, and they are named after the English mathematician John Venn (1834--1923), who used them in the context of propositional logic~\cite{venn_1880}.
Despite the simple definition, many questions about Venn diagrams lead to interesting and challenging mathematical and computational problems, which triggered a long and fruitful line of research devoted to them; see Ruskey and Weston's survey~\cite{MR1668051} and the many beautiful illustrations therein.
These problems touch and connect various areas such as discrete geometry, graph drawing, graph theory, poset theory, coding theory, and enumerative combinatorics.

\begin{figure}[t]
\centerline{
\includegraphics[page=1]{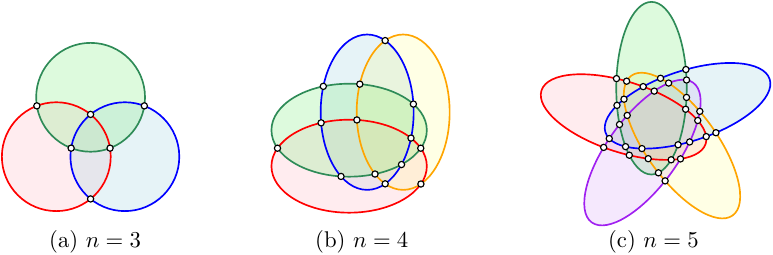}
}
\caption{Simple Venn diagrams for $n=3,4,5$, maximizing the number of crossings.
For $n=3,4,5$ the number of crossings $2^n-2$ equals~$6,14,30$.
The diagrams (a)+(c) are rotationally symmetric.
All diagrams are monotone and convex.}
\label{fig:dia1}
\end{figure}

General constructions of $n$-Venn diagrams, valid for every $n\geq 1$, were provided already by Venn~\cite{venn_1880} and much later by Edwards~\cite{edwards_1989}.
In fact, these two constructions are connected via the well-known binary reflected Gray code, a listing of all $2^n$ binary strings of length~$n$ such that any two consecutive strings differ in a single bit.

Particularly pleasing for the human eye are rotationally \defi{symmetric} Venn diagrams, such as the ones in Figure~\ref{fig:dia1}~(a)+(c) and Figure~\ref{fig:dia2}~(a)+(e).
They exist if and only if $n\geq 2$ is a prime number.
The necessity of this condition was established by Henderson~\cite{MR1532112}, and the sufficiency, i.e., a construction valid for all prime~$n$ was shown by Griggs, Killian and Savage~\cite{MR2034416}.
Their approach builds a symmetric chain partition in the so-called necklace poset.
The problem to find \emph{simple} symmetric $n$-Venn diagrams is open in general.
Solutions for small cases are only known for $n=1,3,5,7,11,13$ \cite{MR3231031}, and for general prime~$n$ a construction is known that guarantees at least half of all crossings to be simple~\cite{MR2114190}.

A \defi{$k$-region} in an $n$-Venn diagram is a region that lies inside of exactly $k$ of the curves, and outside the remaining $n-k$ curves.
A \defi{monotone} diagram is one in which for every $0<k<n$, every $k$-region is adjacent to both a $(k-1)$-region and a $(k+1)$-region.
A \defi{convex} diagram is one in which every curve is convex.
It is not hard to show that convex diagrams are monotone.
Bultena, Gr\"unbaum and Ruskey~\cite{DBLP:conf/cccg/BultenaGR99} also proved the converse, namely that every monotone diagram is isomorphic to a convex one.
Thus, the combinatorial notion of monotonicity completely captures the geometric notion of convexity.

The number of non-isomorphic simple $n$-Venn diagrams for $n=1,\ldots,6$ is $1,1,1,1,20,3\,430\,404$, and the number of monotone simple diagrams is $1,1,1,1,11,32\,255$~\cite{MR1486434,MR1789063,venn_preprint} (OEIS~A386795 and~A390247, respectively).

A well-known conjecture in the area, raised by Peter Winkler~\cite{MR777726} in 1984 and reiterated in~\cite{DBLP:journals/cacm/Winkler12}, was that every simple $n$-Venn diagram can be extended to a simple $(n+1)$-Venn diagram by adding a suitable curve.
Very recently, Winkler's conjecture was disproved by Brenner, Kleist, Mütze, Rieck, and Verciani~\cite{venn_preprint}, who constructed counterexamples to the conjecture for all~$n \geq 6$.
In particular, out of the $3\,430\,404$ many 6-Venn diagrams, 72 are not extendable to a 7-Venn diagram.
Already earlier, Gr\"unbaum~\cite{MR1208440} proposed a variant of Winkler's conjecture, by dropping the requirement for the diagrams to be simple.
Gr\"unbaum's conjecture was settled affirmatively by Chilakamarri, Hamburger and Pippert~\cite{MR1400982}, using a classical theorem in graph theory of Whitney, later generalized by Tutte.

\subsection{Maximizing and minimizing the number of crossings}

\begin{figure}[b!]
\centerline{
\includegraphics[page=2]{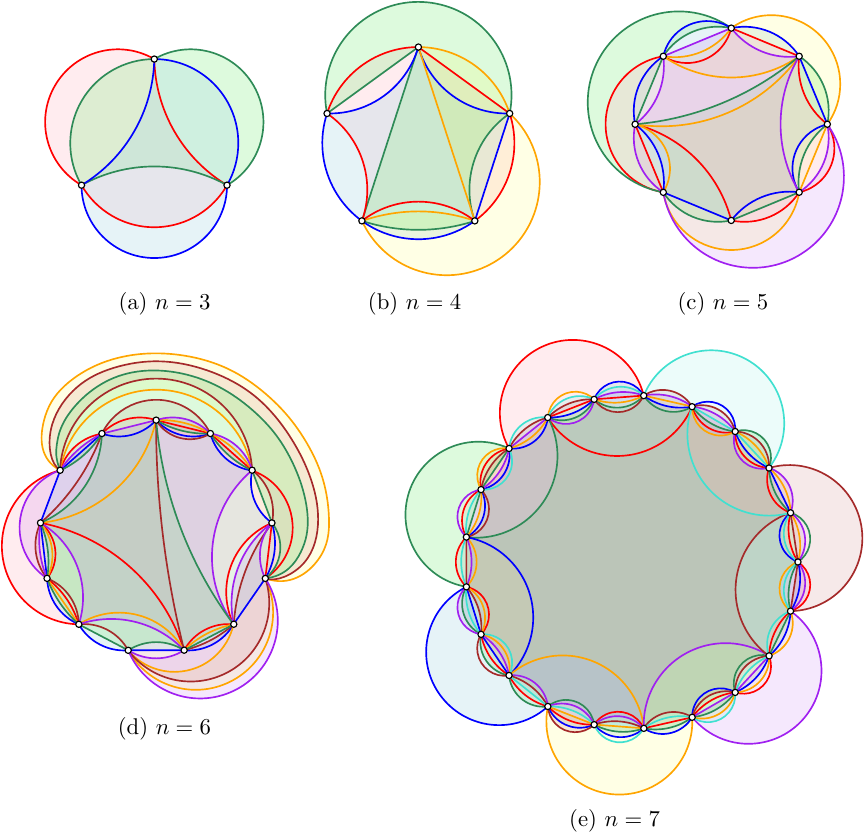}
}
\caption{Non-simple Venn diagrams for $n=3,4,5,6,7$ that minimize the number of crossings.
For those values of~$n$, the number of crossings $L_n=\lceil\frac{2^n-2}{n-1}\rceil$ equals~$3,5,8,13,21$.
Only in the diagrams~(a) and~(e) every crossing involves all $n$ curves.
The diagrams~(a) and~(e) are symmetric.
The diagram~(a) is monotone, while (b)--(e) are not monotone.
None of the diagrams is convex, though (a) can be made convex by changing the radii of the circular arcs.}
\label{fig:dia2}
\end{figure}

In this work, we are particularly interested in the number of crossings in an $n$-Venn diagram.
It is easy to see that simple $n$-Venn diagrams have exactly $2^n-2$ many crossings, and this is the largest possible number of crossings among all $n$-Venn diagrams.
To complement this, Bultena and Ruskey~\cite{MR1649092} considered Venn diagrams with the smallest possible number of crossings.
They proved that any $n$-Venn diagram with $n\geq 2$ has at least
\[ L_n\ass\left\lceil\frac{2^n-2}{n-1}\right\rceil \]
many crossings, and they called diagrams achieving this lower bound \defi{minimum} Venn diagrams.
If $(2^n-2)/(n-1)$ is integral, which happens for $n=2,3,7,19,43,55,127,163,\ldots$ (this is OEIS sequence~A014741 incremented by~1), then this means that every crossing of the diagram involves all $n$ curves.
So far, minimum Venn diagrams are only known for $n\leq 7$; see Figure~\ref{fig:dia2}.
Bultena and Ruskey~\cite{MR1649092} conjectured that minimum $n$-Venn diagrams exist for all $n\geq 8$, and this is also mentioned as an open problem in Ruskey and Weston's survey~\cite{MR1668051}.

As a partial result, Bultena and Ruskey~\cite{MR1649092} proved that among \emph{monotone} $n$-Venn diagrams with $n\geq 2$, the smallest possible number of crossings is $\binom{n}{\lfloor n/2\rfloor}$, and this lower bound is achievable.
Note that $\binom{n}{\lfloor n/2\rfloor}=\frac{2^n}{\sqrt{\pi n/2}}(1+o(1))$, so this is still a $\Theta(\sqrt{n})$-factor away from the general lower bound~$L_n$.

\subsection{Our results}
\label{sec:results}

Our first construction yields $n$-Venn diagrams for the case when $n$ is a power of~2 in which the number of crossings is minimized up to a $(1+o(1))$-factor (as $n\rightarrow \infty$).

\begin{theorem}
\label{thm:min}
There is an $8$-Venn diagram with $40$ crossings, and for every $k\ge 4$ and $n\ass 2^k$, there is an $n$-Venn diagram with exactly
\[ \left(1+\frac{33}{8n}-\frac{2}{2^{n/2}}-\frac{2n}{2^n}\right)\frac{2^n}{n}\le \left(1+\frac{33}{8n}\right)L_n=(1+o(1))L_n \]
many crossings.
\end{theorem}

The 8-Venn diagram with 40 crossings is shown in Figure~\ref{fig:v8-primal}.

\begin{figure}[b!]
\centerline{
\includegraphics{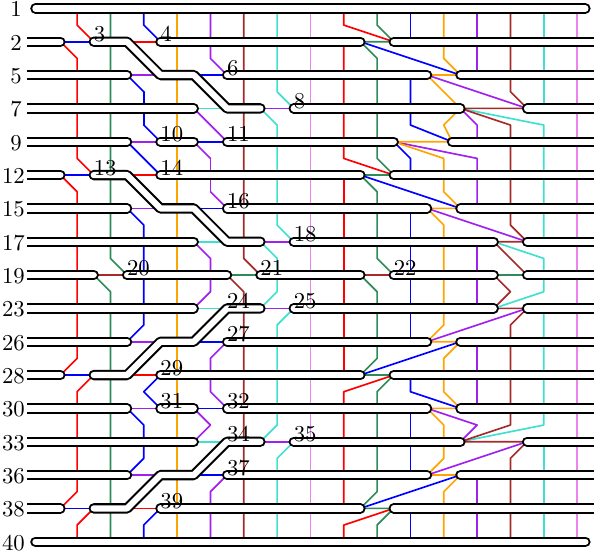}
}
\caption{An 8-Venn diagram with 40 crossings obtained from Theorem~\ref{thm:min}.
The 8 curves are drawn with colors.
Each black/white bubble represents one crossing that has to be contracted to a single point, where the open bubbles at the left and right boundary wrap around at the bottom as if drawn on a cylinder.}
\label{fig:v8-primal}
\end{figure}

For a general number of curves, i.e., when $n$ is not necessarily a power of~2, we apply a doubling construction to obtain an $n$-Venn diagram in which the number of crossings is minimized up to a $(2+o(1))$-factor.

\begin{theorem}
\label{thm:min-all}
For every $0\le m \le 7$ there is an $(8+m)$-Venn diagram with $40\cdot 2^m$ crossings, and for every $k\geq 4$, $n\ass 2^k$ and $0\leq m<n$, there is an $(n+m)$-Venn diagram with exactly
\[ \left(1+\frac{33}{8n}-\frac{2}{2^{n/2}}-\frac{2n}{2^n}\right)\frac{2^{n+m}}{n}\le \left(1+\frac{33}{8n}\right)\frac{n+m}{n}L_{n+m}\le(2+o(1))L_{n+m} \]
many crossings.
\end{theorem}

Figure~\ref{fig:approx} shows the approximation ratio guaranteed by Theorem~\ref{thm:min-all} for all $8\leq n+m\leq 255$.

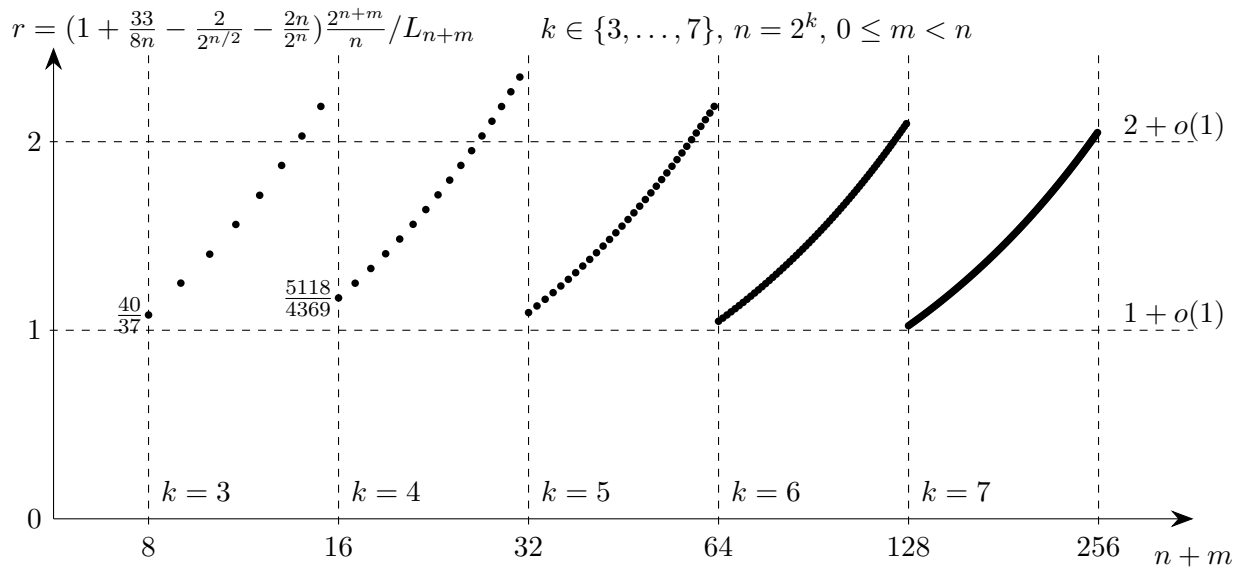
\begin{figure}[h!]
\begin{tikzpicture}[x=25mm,y=25mm]
  \draw[-{Stealth[length=3mm]}] (2.5,0) -- (8.5,0);
  \draw[-{Stealth[length=3mm]}] (2.5,0) -- (2.5,2.5);
  \foreach \x/\xp in {3/8,4/16,5/32,6/64,7/128}
  {\draw [dashed] (\x,-0.03) -- (\x,2.5);
   \node at (\x,-0.15) {$\xp$};
   \node at (\x+0.25,0.15) {$k=\x$};}
  \draw [dashed] (8,-0.03) -- (8,2.5);
  \node at (8,-0.15) {$256$};
  \draw [dashed] (2.48,1) -- (8.5,1);
  \draw [dashed] (2.48,2) -- (8.5,2);
  \foreach \y in {0,1,2}
  {\node at (2.4,\y) {$\y$};}
  \node at (3.5,2.6) {$r=(1+\frac{33}{8n}-\frac{2}{2^{n/2}}-\frac{2n}{2^n})\frac{2^{n+m}}{n}/L_{n+m}$};
  \node at (6.2,2.6) {$k\in\{3,\ldots,7\}$, $n=2^k$, $0\leq m<n$};
  \node at (8.5,-0.2) {$n+m$};
  \node at (2.9,1.08) {$\frac{40}{37}$};
  \node at (3.84,1.17) {$\frac{5118}{4369}$};
  \node at (8.4,1.08) {$1+o(1)$};
  \node at (8.4,2.08) {$2+o(1)$};
  \foreach \x/\y in {3./1.08108, 3.16993/1.25, 3.32193/1.40351, 3.45943/1.56098, 3.58496/1.71582, 3.70044/1.87408, 3.80735/2.03013, 3.90689/2.1871, 4./1.17144, 4.08746/1.24951, 4.16993/1.32754, 4.24793/1.40571, 4.32193/1.48377, 4.39232/1.56188, 4.45943/1.63998, 4.52356/1.71808, 4.58496/1.79617, 4.64386/1.87427, 4.70044/1.95236, 4.75489/2.03046, 4.80735/2.10855, 4.85798/2.18665, 4.90689/2.26474, 4.9542/2.34283, 5./1.0936, 5.04439/1.12888, 5.08746/1.16415, 5.12928/1.19943, 5.16993/1.23471, 5.20945/1.26999, 5.24793/1.30526, 5.2854/1.34054, 5.32193/1.37582, 5.35755/1.41109, 5.39232/1.44637, 5.42626/1.48165, 5.45943/1.51693, 5.49185/1.5522, 5.52356/1.58748, 5.55459/1.62276, 5.58496/1.65804, 5.61471/1.69331, 5.64386/1.72859, 5.67243/1.76387, 5.70044/1.79915, 5.72792/1.83442, 5.75489/1.8697, 5.78136/1.90498, 5.80735/1.94026, 5.83289/1.97553, 5.85798/2.01081, 5.88264/2.04609, 5.90689/2.08136, 5.93074/2.11664, 5.9542/2.15192, 5.97728/2.1872, 6./1.04782, 6.02237/1.06445, 6.04439/1.08109, 6.06609/1.09772, 6.08746/1.11435, 6.10852/1.13098, 6.12928/1.14761, 6.14975/1.16425, 6.16993/1.18088, 6.18982/1.19751, 6.20945/1.21414, 6.22882/1.23077, 6.24793/1.24741, 6.26679/1.26404, 6.2854/1.28067, 6.30378/1.2973, 6.32193/1.31393, 6.33985/1.33057, 6.35755/1.3472, 6.37504/1.36383, 6.39232/1.38046, 6.40939/1.39709, 6.42626/1.41373, 6.44294/1.43036, 6.45943/1.44699, 6.47573/1.46362, 6.49185/1.48026, 6.50779/1.49689, 6.52356/1.51352, 6.53916/1.53015, 6.55459/1.54678, 6.56986/1.56342, 6.58496/1.58005, 6.59991/1.59668, 6.61471/1.61331, 6.62936/1.62994, 6.64386/1.64658, 6.65821/1.66321, 6.67243/1.67984, 6.6865/1.69647, 6.70044/1.7131, 6.71425/1.72974, 6.72792/1.74637, 6.74147/1.763, 6.75489/1.77963, 6.76818/1.79626, 6.78136/1.8129, 6.79442/1.82953, 6.80735/1.84616, 6.82018/1.86279, 6.83289/1.87943, 6.84549/1.89606, 6.85798/1.91269, 6.87036/1.92932, 6.88264/1.94595, 6.89482/1.96259, 6.90689/1.97922, 6.91886/1.99585, 6.93074/2.01248, 6.94251/2.02911, 6.9542/2.04575, 6.96578/2.06238, 6.97728/2.07901, 6.98868/2.09564, 7./1.02416, 7.01123/1.03223, 7.02237/1.04029, 7.03342/1.04836, 7.04439/1.05642, 7.05528/1.06448, 7.06609/1.07255, 7.07682/1.08061, 7.08746/1.08868, 7.09803/1.09674, 7.10852/1.1048, 7.11894/1.11287, 7.12928/1.12093, 7.13955/1.129, 7.14975/1.13706, 7.15987/1.14513, 7.16993/1.15319, 7.17991/1.16125, 7.18982/1.16932, 7.19967/1.17738, 7.20945/1.18545, 7.21917/1.19351, 7.22882/1.20158, 7.2384/1.20964, 7.24793/1.2177, 7.25739/1.22577, 7.26679/1.23383, 7.27612/1.2419, 7.2854/1.24996, 7.29462/1.25803, 7.30378/1.26609, 7.31288/1.27415, 7.32193/1.28222, 7.33092/1.29028, 7.33985/1.29835, 7.34873/1.30641, 7.35755/1.31448, 7.36632/1.32254, 7.37504/1.3306, 7.3837/1.33867, 7.39232/1.34673, 7.40088/1.3548, 7.40939/1.36286, 7.41785/1.37093, 7.42626/1.37899, 7.43463/1.38705, 7.44294/1.39512, 7.45121/1.40318, 7.45943/1.41125, 7.46761/1.41931, 7.47573/1.42738, 7.48382/1.43544, 7.49185/1.4435, 7.49985/1.45157, 7.50779/1.45963, 7.5157/1.4677, 7.52356/1.47576, 7.53138/1.48383, 7.53916/1.49189, 7.54689/1.49995, 7.55459/1.50802, 7.56224/1.51608, 7.56986/1.52415, 7.57743/1.53221, 7.58496/1.54028, 7.59246/1.54834, 7.59991/1.5564, 7.60733/1.56447, 7.61471/1.57253, 7.62205/1.5806, 7.62936/1.58866, 7.63662/1.59673, 7.64386/1.60479, 7.65105/1.61285, 7.65821/1.62092, 7.66534/1.62898, 7.67243/1.63705, 7.67948/1.64511, 7.6865/1.65318, 7.69349/1.66124, 7.70044/1.6693, 7.70736/1.67737, 7.71425/1.68543, 7.7211/1.6935, 7.72792/1.70156, 7.73471/1.70963, 7.74147/1.71769, 7.74819/1.72575, 7.75489/1.73382, 7.76155/1.74188, 7.76818/1.74995, 7.77479/1.75801, 7.78136/1.76608, 7.7879/1.77414, 7.79442/1.7822, 7.8009/1.79027, 7.80735/1.79833, 7.81378/1.8064, 7.82018/1.81446, 7.82655/1.82253, 7.83289/1.83059, 7.8392/1.83865, 7.84549/1.84672, 7.85175/1.85478, 7.85798/1.86285, 7.86419/1.87091, 7.87036/1.87897, 7.87652/1.88704, 7.88264/1.8951, 7.88874/1.90317, 7.89482/1.91123, 7.90087/1.9193, 7.90689/1.92736, 7.91289/1.93542, 7.91886/1.94349, 7.92481/1.95155, 7.93074/1.95962, 7.93664/1.96768, 7.94251/1.97575, 7.94837/1.98381, 7.9542/1.99187, 7.96/1.99994, 7.96578/2.008, 7.97154/2.01607, 7.97728/2.02413, 7.98299/2.0322, 7.98868/2.04026, 7.99435/2.04832}
  { \draw (\x,\y) node[circle,inner sep=1pt,fill=black] {}; }
\end{tikzpicture}
\caption{Log-linear plot of the ratio~$r$ of the number of crossings of the Venn diagrams obtained from Theorem~\ref{thm:min-all} and the lower bound.}
\label{fig:approx}
\end{figure}

In particular, Theorem~\ref{thm:min-all} gives $n$-Venn diagrams with the smallest known number of crossings for all $n\ge 8$; see Table~\ref{tab:values}.

\begin{table}[h!]
\caption{The number of crossings in the Venn diagrams obtained from Theorem~\ref{thm:min-all} versus the lower bound, and the minimum numbers for monotone diagrams.}
\label{tab:values}
\setlength{\tabcolsep}{4pt}
\makebox[0cm]{ 
\begin{tabular}{r|cccccccccccccccc}
$n$ & 1 & 2 & 3 & 4 & 5 & 6 & 7 & 8 & 9 & 10 & 11 & 12 & 13 & 14 & 15 & 16 \\ \hline
lower bound $L_n=\left\lceil\frac{2^n-2}{n-1}\right\rceil$ \cite{MR1649092} & 0 & 2 & 3 & 5 & 8 & 13 & 21 & 37 & 64 & 114 & 205 & 373 & 683 & 1261 & 2341 & 4369 \\
Figure~\ref{fig:dia2} & 0 & 2 & 3 & 5 & 8 & 13 & 21 & & & & & & & & & \\
Theorem~\ref{thm:min-all} & & & & & & & & 40 & 80 & 160 & 320 & 640 & 1280 & 2560 & 5120 & 5118 \\
monotone $\binom{n}{\lfloor n/2 \rfloor}$ \cite{MR1649092} & 0 & 2 & 3 & 6 & 10 & 20 & 35 & 70 & 126 & 252 & 462 & 924 & 1716 & 3432 & 6435 & 12870 \\
\end{tabular}
}
\end{table}

\subsection{Duals of Venn diagrams}
\label{sec:dual}

Venn diagrams are conveniently studied by considering their dual graph; see Figure~\ref{fig:dual}.
Specifically, we consider the \defi{$n$-dimensional hypercube~$Q_n$}, or \defi{$n$-cube} for short, the graph formed by all subsets of~$[n]\ass\{1,2,\ldots,n\}$, with an edge between any two sets that differ in a single element.
For an edge $e=\{x,x\cup\{i\}\}$ of $Q_n$, we refer to $i$ as the \defi{direction} of the edge~$e$.

The dual graph~$Q(D)$ of an $n$-Venn diagram~$D$ satisfies the following properties:
\begin{enumerate}[label=\protect\circled{\arabic*},leftmargin=8mm]
\item It is a spanning subgraph of~$Q_n$, i.e., all $2^n$ vertices are present.
Specifically, the vertex in~$Q(D)$ corresponding to a region of the diagram~$D$ is the set of all indices of curves that contain this region in their interior.
\item It is a plane graph, i.e., it is drawn in the plane without edge crossings, such that each face has even length~$2\ell$ for some integer~$2\leq \ell\leq n$ and contains exactly two edges of $\ell$ distinct directions.
These $\ell$ directions correspond to the $\ell$ curves intersecting in the crossing corresponding to the face.
\item For every~$j\in[n]$, the two subgraphs of~$Q(D)$ induced by all vertices~$x$ with $j\in x$, and with $j\notin x$, respectively, are connected.
This corresponds to the $j$th curve being simple.
\end{enumerate}
Conversely, the dual of any subgraph of~$Q_n$ satisfying these properties is an $n$-Venn diagram.

\begin{figure}
\centerline{
\includegraphics[page=3]{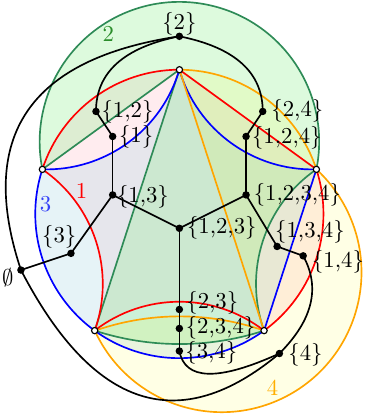}
}
\caption{The dual graph of the 4-Venn diagram from Figure~\ref{fig:dia2}~(b) is a spanning subgraph of the 4-cube.}
\label{fig:dual}
\end{figure}

Clearly, the number of crossings of the diagram~$D$ equals the number of faces of the dual graph~$Q(D)$.
If $D$ is a minimum $n$-Venn diagram, then almost all crossings involve all $n$ curves, i.e., in the dual graph~$Q(D)$, almost all faces have length~$2n$.

\subsection{Proof ideas}
\label{sec:ideas}

Our proofs of Theorems~\ref{thm:min} and~\ref{thm:min-all} are constructive, and we proceed to give an informal sketch of the main steps of these constructions.

Using the characterization presented in the previous section, we build $n$-Venn diagrams by constructing planar spanning subgraphs of the hypercube~$Q_n$ satisfying conditions~\circled{1}--\circled{3}, with the goal of minimizing the number of faces.

For $n=2^k$ we start with a partition of~$Q_n$ into isometric cycles of length~$2n$ that is derived from a partition of~$Q_{n-1}$ into isometric paths found by Ramras~\cite{MR1152451}.
These cycles are translates of a single isometric cycle of length~$2n$ by a certain linear subspace of $\mathbb{Z}_2^n$ of dimension $d\ass 2^k-k-1$.
For our construction we choose a slightly different basis of this space, denoted by~$C_k$, than originally used by Ramras.

In the first step of the construction, we embed the isometric cycles of the partition into the plane concentrically in an order that is specified by a  certain Hamiltonian path~$R$ of~$Q_d$ (see Figure~\ref{fig:v8-dual}).
Each vertex of~$Q_d$ represents coefficients of a linear combination of the basis~$C_k$, so consecutive concentric cycles in this order differ only by one element of the basis~$C_k$.
This guarantees that there are edges between consecutive cycles that create mainly large faces, but also some 6-faces.
These edges between pairs of consecutive cycles are added in the second step.

In the last step, some cycle edges shared by 6-faces are removed to further reduce the total number of faces.
For this purpose we specify a Hamiltonian path~$R$ that contains long runs in its flip sequence.
A \defi{run} is a contiguous increasing or decreasing subsequence with increment or decrement~1, respectively.
Interestingly, our construction of such long-run Hamiltonian cycles uses the same partition into isometric cycles as the aforementioned Venn diagram construction.

For values of~$n$ that are not powers of~2, we use a straightforward doubling construction (see Lemma~\ref{lem:doubling}) to derive $n$-Venn diagrams with relatively few crossings from the ones of the next smaller power of~2 constructed as described before.
This is why the number of crossings increasingly deviates from the lower bound~$L_n$ the farther the distance from the next smaller power of~2 gets, until the next larger power of~2 resets the process, which explains the behavior seen in Figure~\ref{fig:approx}.

\section{Isometric partitions of hypercubes}

\subsection{Preliminaries}

Recall that $[n]=\{1,2,\ldots,n\}$.
We also define $2^{[n]}\ass\{x\mid x\seq [n]\}$.
For $x\seq [n]$ we write $\ol{x}\ass[n]\setminus x$ for the complement of~$x$ w.r.t.\ the ground set~$[n]$.
Given a vertex~$x$ of~$Q_n$, we refer to $\ol{x}$ as the \defi{antipodal} vertex of~$x$.
For $x\seq [n]$ and an integer~$k\geq 1$ we define $k\cdot x\ass\{ki\mid i\in x\}$.
All these operations thread in the natural way over sets and sequences.
For example, for a set~$X\seq 2^{[n]}$ and an integer~$k\geq 1$ we have $k\cdot X=\{k\cdot x\mid x\in X\}$.

For $x,y\seq [n]$ we define the symmetric difference as $x\oplus y\ass(x\setminus y)\cup (y\setminus x)$.
For any set $X\seq 2^{[n]}$ we define the \defi{span}~of~$X$ by $\langle X\rangle\ass\{x_1\oplus x_2\oplus \cdots\oplus x_t\mid t\in\{0,\ldots,n\}\text{ and } x_i\in X\text{ for all }i\in[t]\}$.

We sometimes refer to an edge~$e=\{x,x\oplus\{i\}\}$ of direction~$i$ in~$Q_n$ as an \defi{$i$-edge}.
The \defi{flip sequence} of a path or cycle~$P$ in~$Q_n$, denoted by~$\sigma(P)$, is the sequence of directions of edges along~$P$.
We write $|\sigma(P)|$ for the length of this sequence.
For example, the flip sequence of the path~$P=(\{1,2\},\{1,2,3\},\{1,3\},\{3\},\{2,3\})$ in~$Q_3$ is $\sigma(P)=(3,2,1,2)$ and we have $|\sigma(P)|=4$.

A subgraph~$H$ of a graph~$G$ is \defi{isometric} if it preserves distances, i.e., $d_H(u,v)=d_G(u,v)$ for any two vertices $u,v \in V(H)$.
In particular, a path in the hypercube~$Q_n$ is isometric if and only if it contains no two edges of the same direction.
Similarly, a cycle in~$Q_n$ is isometric if and only the edges of the same direction come in pairs that lie oppositely on the cycle.
If the isometric cycle has length~$2n$, then it is the union of an isometric path and its corresponding antipodal path.

\subsection{Partition into isometric paths}
\label{sec:part}

Ramras~\cite{MR1152451} described a partition of the hypercube~$Q_{n-1}$, for $n=2^k$ and $k\geq 1$, into isometric paths.
For this we consider the sequence of sets~$(B_k)_{k\geq 1}$ defined recursively by
\begin{equation}
\label{eq:Bk}
\begin{split}
B_1&\ass\emptyset, \text{ and } \\
B_k&\ass B_{k-1}\cup\{\{1,2^{k-1}+1\},\{2,2^{k-1}+2\},\ldots,\{2^{k-1}-1,2^k-1\}\} \text{ for } k\geq 2.
\end{split}
\end{equation}
According to this definition, the first few sets are
\[
\begin{split}
B_1&=\emptyset, \\
B_2&=\{\{1,3\}\}, \\
B_3&=\{\{1,3\},\{1,5\},\{2,6\},\{3,7\}\}, \\
B_4&=\{\{1,3\},\{1,5\},\{2,6\},\{3,7\},\{1,9\},\{2,10\},\{3,11\},\{4,12\},\{5,13\},\{6,14\},\{7,15\}\}.
\end{split}
\]
The sets in~$B_k$, viewed as binary (characteristic) vectors, are linearly independent and hence they form the basis of a linear subspace $\langle B_k \rangle$ of dimension $|B_k|=2^k-k-1$ of the space $\mathbb{Z}_2^n$.
For any $x\in \langle B_k\rangle$, Ramras~\cite{MR1152451} defines a path $P(x)$ of length~$n-1$ in~$Q_{n-1}$ by
\begin{equation}
\label{eq:Px}
P(x)\ass (x,x\oplus \{1\},x\oplus \{1,2\},\ldots,x\oplus\{1,2,\ldots,n-1\}).
\end{equation}
The flip sequence along this path is $\sigma(P)=(1,2,\ldots,n-1)$, and the end vertices of~$P(x)$ are antipodal in~$Q_{n-1}$, i.e., the last vertex $x\oplus\{1,2,\ldots,n-1\}$ is the complement of the first vertex~$x$ w.r.t.\ the ground set~$[n-1]$.
The following theorem is illustrated on the left hand side of Figure~\ref{fig:iso-part}.

\begin{theorem}[{\cite[Thm.~2.5]{MR1152451}}]
\label{thm:iso-part-paths}
For any $k\geq 1$ and $n\ass 2^k$, the isometric paths $P(x)$ of length~$n-1$ defined in~\eqref{eq:Px} for all $x\in\langle B_k\rangle$ form a partition of the vertex set of~$Q_{n-1}$.
\end{theorem}

\begin{figure}[h!]
\centerline{
\includegraphics[page=1]{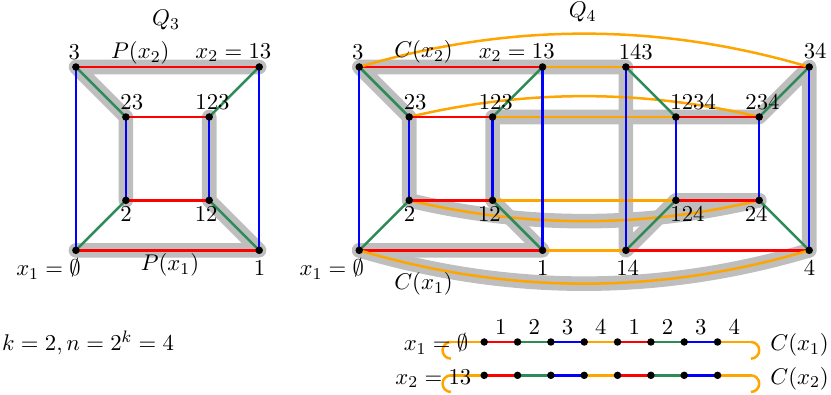}
}
\caption{Illustration of Theorem~\ref{thm:iso-part-paths} (left) and Corollary~\ref{cor:iso-part-cycles} (right) for the case $k=2$.
The bottom right shows a schematic illustration of the cycles showing only the starting vertices on the left and the flip sequences along each cycle.}
\label{fig:iso-part}
\end{figure}

\subsection{Finding another basis for $\langle B_k\rangle$}

For our purposes, we choose a different basis of the subspace $\langle B_k\rangle$, namely one that contains a larger number of $2$-sets $\{a,b\}$ with $|b-a|=2$.
We define
\begin{equation}
\label{eq:OkCk}
\begin{split}
C_1&\ass O_1\ass \emptyset,\\
O_k&\ass \{\{1,3\},\{3,5\},\ldots,\{2^k-3,2^k-1\}\} \text{ for $k\geq 2$, and} \\
C_k& \ass O_k\cup 2C_{k-1} \text{ for } k\geq 2.
\end{split}
\end{equation}
According to this definition, the first few sets are
\[
\begin{split}
O_1&=\emptyset, \\
O_2&=\{\{1,3\}\}, \\
O_3&=\{\{1,3\},\{3,5\},\{5,7\}\}, \\
O_4&=\{\{1,3\},\{3,5\},\{5,7\},\{7,9\},\{9,11\},\{11,13\},\{13,15\}\},
\end{split}
\]
and
\[
\label{eq:C14}
\begin{split}
C_1&=\emptyset, \\
C_2&=\{\{1,3\}\}, \\
C_3&=\{\{1,3\},\{3,5\},\{5,7\},\{2,6\}\}, \\
C_4&=\{\{1,3\},\{3,5\},\{5,7\},\{7,9\},\{9,11\},\{11,13\},\{13,15\},\{2,6\},\{6,10\},\{10,14\},\{4,12\}\}.
\end{split}
\]

Note that $|b-a|=2$ for every $\{a,b\}\in O_k$ and furthermore, $|O_k|=2^{k-1}-1$ and $|C_k|=2^k-k-1$.

\begin{lemma}
\label{lem:Cdiff}
For any two distinct sets $\{a,b\},\{a',b'\}\in C_k$ with $a<b$ and $a'<b'$ we have $a \ne a'$ and $b\ne b'$.
\end{lemma}

\begin{proof}
This follows from the definition~\eqref{eq:OkCk}, noting that the property trivially holds for~$O_k$, that it is preserved under multiplication by~2, and that all elements of~$O_k$ contain only odd numbers, whereas all elements of~$2C_{k-1}$ contain only even numbers.
\end{proof}

\begin{lemma}
\label{lem:BkCk}
The sets~$B_k$ and~$C_k$ defined in~\eqref{eq:Bk} and~\eqref{eq:OkCk}, respectively, satisfy $\langle B_k\rangle=\langle C_k\rangle$ for any $k\geq 1$.
\end{lemma}

\begin{proof}
Unrolling the recursive definition~\eqref{eq:OkCk}, we obtain
\begin{equation}
\label{eq:CkOk-union}
 C_k=O_k\cup 2O_{k-1}\cup 4O_{k-2}\cup\cdots\cup 2^{k-1}O_1=\bigcup_{i=0}^{k-1} 2^i O_{k-i}.
\end{equation}
From~\eqref{eq:OkCk} we directly obtain $O_{k-1}\seq O_k$, and combining this with~\eqref{eq:CkOk-union} proves that $C_{k-1}\seq C_k$.

We first show by induction on~$k$ that if $\{a,b\}\in B_k$ with $a<b$, then $\{a,b\}\in \langle C_k \rangle$.
The induction basis $k=1$ is trivial.
For the induction step we assume that $k\geq 2$.
There are two possible cases, namely $\{a,b\}\in B_{k-1}$ and $\{a,b\}\in B_k\setminus B_{k-1}$.
If $\{a,b\}\in B_{k-1}$ then we have $\{a,b\}\in \langle C_{k-1}\rangle$ by induction, and since $C_{k-1}\subseteq C_k$ as argued before, we indeed obtain that $\{a,b\}\in\langle C_k\rangle$.
On the other hand, if $\{a,b\}\in B_k\setminus B_{k-1}$ then factoring out all 2s from~$a$ gives $\{a,b\}=2^i\{j,j+2^{k-1-i}\}$ for some odd $j\geq 1$ and $0\le i<k-1$ such that $j+2^{k-1-i}<2^{k-i}$.
Since clearly $\{j,j+2^{k-1-i}\} \in \langle O_{k-i}\rangle$, it follows from \eqref{eq:CkOk-union} that $\{a,b\}\in \langle C_k \rangle$.

We have shown that $\langle B_k\rangle \seq \langle C_k\rangle$.
Since the sets in~$B_k$ are linearly independent and $|B_k|=|C_k|$, it follows that $\langle B_k\rangle =\langle C_k\rangle$.
\end{proof}

\subsection{Partition into isometric cycles}

From Theorem~\ref{thm:iso-part-paths} it follows that
\begin{equation}
\label{eq:Cx}
C(x)\ass (P(x),\ol{P(x)}\cup\{n\})
\end{equation}
is an isometric cycle of length~$2n$ in~$Q_n$, $n=2^k$, with flip sequence $\sigma(C(x))=(1,2,\ldots,n,1,2,\ldots,n)$, and furthermore, the cycles $C(x)$ for all $x\in\langle B_k\rangle$ form a partition of the vertex set of~$Q_n$.
Combining this with Lemma~\ref{lem:BkCk} we obtain the following isometric partition into cycles, which is the starting point of our constructions later on; see the right hand side of Figure~\ref{fig:iso-part}.

\begin{corollary}
\label{cor:iso-part-cycles}
For any $k\geq 1$ and $n\ass 2^k$, the isometric cycles $C(x)$ of length~$2n$ defined in~\eqref{eq:Cx} for all $x\in\langle C_k\rangle$ form a partition of the vertex set of~$Q_n$.
\end{corollary}

The following lemma describes the edges between the cycles~$C(x)$, and it is illustrated in Figure~\ref{fig:cross-edges}.
The proof is elementary, and we omit it.

\begin{figure}[t!]
\centerline{
\includegraphics[page=2]{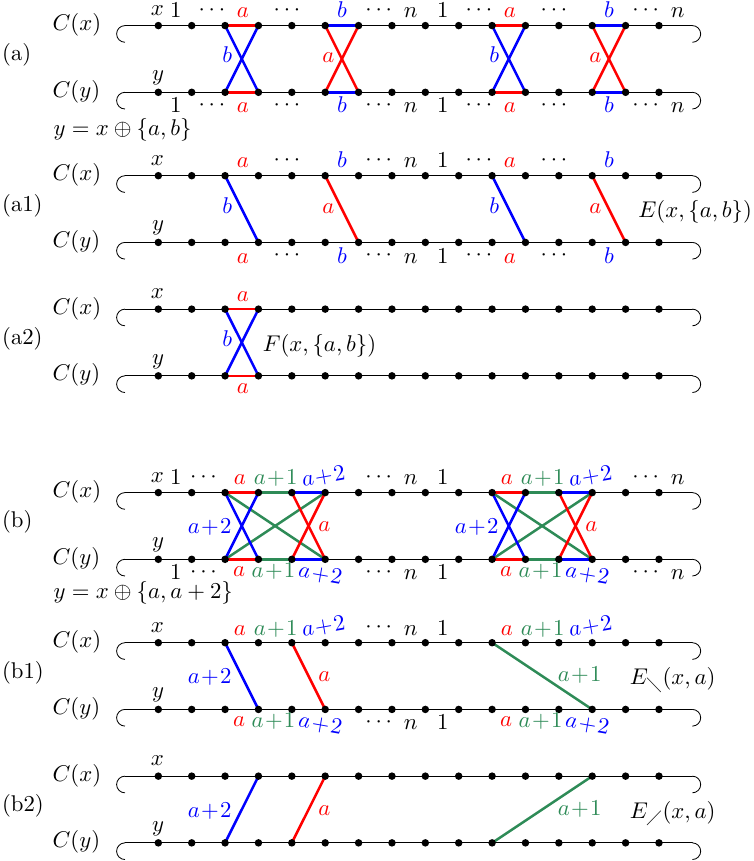}
}
\caption{Illustration of Lemma~\ref{lem:cross-edges} and definitions of the edge sets $E(x,\{a,b\})$, $F(x,\{a,b\})$, $E_{\diagdown}(x,a)$, and $E_{\diagup}(x,a)$.}
\label{fig:cross-edges}
\end{figure}

\begin{lemma}
\label{lem:cross-edges}
Let $k\geq 2$ and $n\ass 2^k$, and let $x,y\in\langle C_k\rangle$ be such that $y=x\oplus \{a,b\}$ with $a<b$.
Then the cycles~$C(x)$ and~$C(y)$ have the following edges between them in $Q_n$:
\begin{enumerate}[label=(\roman*),topsep=0mm,itemsep=0mm,leftmargin=8mm]
\item From the start/end vertex of each of the two $a$-edges of~$C(x)$, there is a $b$-edge to the end/start vertex of the corresponding $a$-edge of~$C(y)$.
\item From the start/end vertex of each of the two $b$-edges of~$C(x)$, there is an $a$-edge to the end/start vertex of the corresponding $b$-edge of~$C(y)$.
\end{enumerate}
If $b=a+2$, then in addition the following edges are present:
\begin{enumerate}[label=(\roman*),topsep=0mm,itemsep=0mm,leftmargin=8mm]
\setcounter{enumi}{2}
\item From the start/end vertex of each of the two $a$-/$(a+2)$-edges of~$C(x)$, there is an $(a+1)$-edge to the end/start vertex of the corresponding $(a+2)$-/$a$-edge of~$C(y)$.
\end{enumerate}
\end{lemma}

For $b>a+2$ we write $E(x,\{a,b\})$ for the set of four edges shown in Figure~\ref{fig:cross-edges}~(a1).
Furthermore, for $b=a+2$ we write $E_{\diagdown}(x,a)$ and~$E_{\diagup}(x,a)$ for the two sets of three edges shown in parts~(b1) and~(b2) of the figure, respectively.
Lastly, we write $F(x,\{a,b\})$ for the 4-cycle shown in part~(a2) of the figure.

\section{Long-run Gray code construction}
\label{sec:run}

Given a path or cycle~$P$ in~$Q_n$ and integer~$\rho\in[n]$, an \defi{increasing} or \defi{decreasing $\rho$-run} in its flip sequence~$\sigma(P)$ is a contiguous subsequence $(a,a+1,a+2,\ldots,a+\ell)$ or $(a+\ell,a+\ell-1,\ldots,a+1,a)$ such that $a+\ell\leq \rho$.
In words, it is a sequence of values that are increasing or decreasing (with increment~1 or decrement~1, respectively) such that all flipped directions are at most~$\rho$.
The parameter $\ell\geq 0$, which is one less than the length of the subsequence, is called the \defi{length} of the run.

A $\rho$-run is \defi{maximal} if it cannot be extended, i.e., if it is not contained in another $\rho$-run.
Note that two maximal $\rho$-runs may overlap in at most one element, namely the last element of an increasing run and the first element of a decreasing run, or vice versa.
Furthermore, note that a maximal run of length~0 does not overlap with any other runs, and a maximal run of length~1 does not overlap with two other runs in the first and last element, as this would create a flip sequence $a+1,a,a+1,a$ or $a,a+1,a,a+1$, i.e., a 4-cycle, which is impossible.
A \defi{$\rho$-run partition} of~$\sigma(P)$ is obtained by considering all maximal $\rho$-runs in~$\sigma(P)$, and removing, from every pair of consecutive overlapping runs, the element in which they overlap from one of the two runs.
For example, the flip sequence $\sigma(P)=(3,1,2,3,2,1,2,4,3,2,4)$ has the maximal 3-runs
$
(\overbracket[0.5pt][2pt]{3},\lefteqn{\overbracket[0.5pt][2pt]{\phantom{1,2,3}}}1,2,\lefteqn{\underbracket[0.5pt][2pt]{\phantom{3,2,1}}}3,2,\lefteqn{\overbracket[0.5pt][2pt]{\phantom{1,2}}}1,2,\bcancel{4},\underbracket[0.5pt][2pt]{3,2},\bcancel{4}),
$
where increasing and decreasing runs are marked by overbrackets and underbrackets, respectively, and the cancelled elements are not contained in any 3-run.
Thus we obtain $(\mybox{3},\mybox{1,2},\mybox{3,2,1},\mybox{2},\bcancel{4},\mybox{3,2},\bcancel{4})$ as a 3-run partition of~$\sigma(P)$, where the boxes indicate the runs.
Similarly, $(\mybox{3},\mybox{1,2,3},\mybox{2,1},\mybox{2},\bcancel{4},\mybox{3,2},\bcancel{4})$, $(\mybox{3},\mybox{1,2},\mybox{3,2},\mybox{1,2},\bcancel{4},\mybox{3,2},\bcancel{4})$ and $(\mybox{3},\mybox{1,2,3},\mybox{2},\mybox{1,2},\bcancel{4},\mybox{3,2},\bcancel{4})$ are also valid 3-run partitions.

Although there may be different $\rho$-run partitions, their number and total length, i.e., sum of lengths of all runs, is the same, and we denote these quantities by \defi{$\nu_\rho(P)$} and \defi{$\lambda_\rho(P)$}.
Note that the number of entries of~$\sigma(P)$ exceeding~$\rho$ equals $|\sigma(P)|-\nu_\rho(P)-\lambda_\rho(P)$.
In the example from before we have $\nu_3(P)=4$ and $\lambda_3(P)=3$.

Our first aim is to find a Hamiltonian path~$P$ in the hypercube~$Q_n$ for $n=2^k$ for which the quantity $\nu_{n-1}(P)+2\lambda_{n-1}(P)$ is as large as possible, i.e., we want to maximize the number of flipped directions contained in long $(n-1)$-runs.
We shall see that maximizing this quantity corresponds to minimizing the number of crossings in the Venn diagrams constructed in the next section (see Lemma~\ref{lem:min}).

\begin{lemma}
\label{lem:run}
Let $k\geq 2$ and $n\ass 2^k$.
There is a Hamiltonian path $P$ in~$Q_n$ that satisfies
\[ \nu_{n-1}(P)=\begin{cases}
               6& \text{if } k=2, \\
               \frac{17}{8}\cdot\frac{2^n}{n}-4 & \text{if } k\ge 3,
               \end{cases} \text{   and   }
\lambda_{n-1}(P)=\begin{cases}
                 8 & \text{if } k=2, \\
                 2^n-\frac{25}{8}\cdot \frac{2^n}{n}+4 & \text{if } k\ge 3.
                 \end{cases} \]
\end{lemma}

\begin{proof}
For $k=2$ we take the Hamiltonian path~$P$ with $\sigma(P)=(\mybox{1,2,3},\mybox{2},\mybox{1,2,3},\bcancel{4},\mybox{3,2,1},\mybox{2},\mybox{3,2,1})$, proving that $\nu_3(P)=6$ and $\lambda_3(P)=8$; see the top part of Figure~\ref{fig:run}.

\begin{figure}[t!]
\centerline{
\includegraphics[page=4]{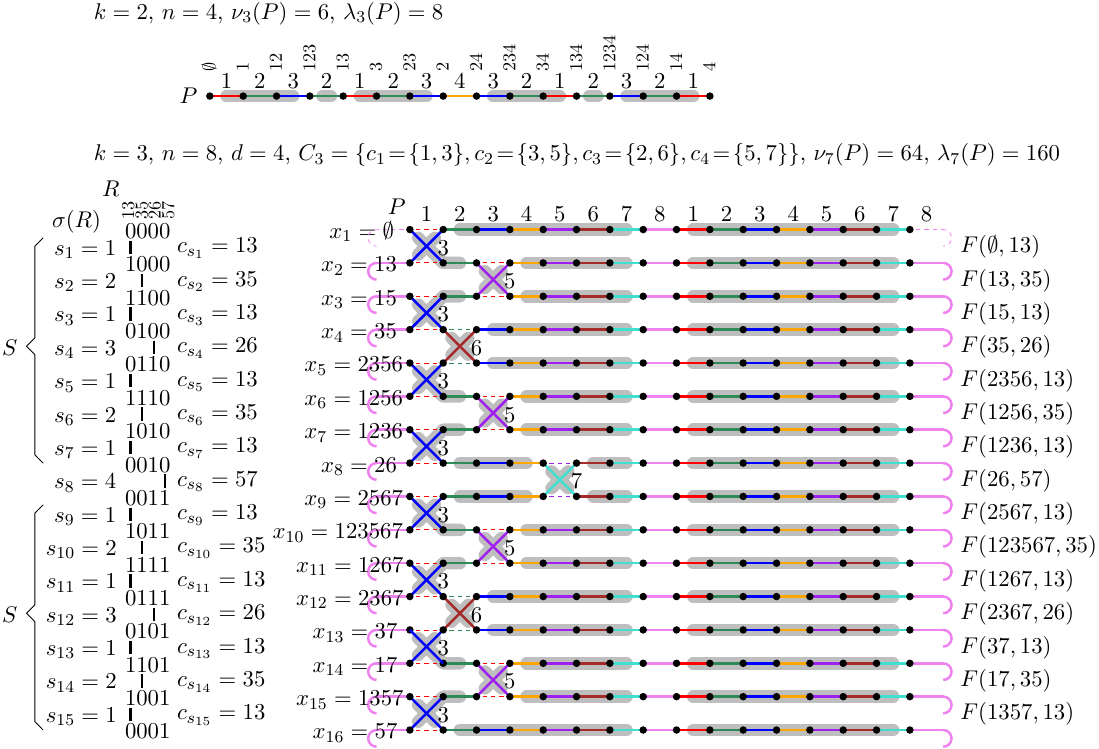}
}
\caption{Illustration of the proof of Lemma~\ref{lem:run}.
The edges that are removed while taking the symmetric difference are dashed, including the wrap-around edge of the first cycle.
The subsequences highlighted in gray are $(n-1)$-runs.}
\label{fig:run}
\end{figure}

For $k\ge 3$ we construct a Hamiltonian path~$P$ in~$Q_n$ as follows; see the bottom part of Figure~\ref{fig:run}.
We start with the partition of~$Q_n$ into cycles given by Corollary~\ref{cor:iso-part-cycles}.
Formally, we define the cycle factor $\cC\ass \bigcup_{x\in\langle C_k\rangle}C(x)$.
Each of the cycles of~$\cC$ has length~$2n$ and the total number of cycles is~$2^n/(2n)=2^{2^k}/2^{k+1}=2^{2^k-k-1}=2^d$ for $d\ass 2^k-k-1$.
Recall that $|C_k|=2^k-k-1=d$.

Let $(c_1,\ldots,c_d)$ be a sequence obtained by sorting all elements of the set~$C_k$ such that $c_1=\{1,3\}, c_2=\{3,5\}, c_3=\{2,6\}$, and the remaining elements appear in arbitrary order.
We take a Hamiltonian path~$R$ in~$Q_d$ whose flip sequence $\sigma(R)\assr(s_1,\ldots,s_{2^d-1})$ alternates copies of the subsequence $S\ass (1,2,1,3,1,2,1)$ with single flips $s_{8i}\ge 4$ for $i\geq 1$, starting with~$S$, i.e., $\sigma(R)=(S,s_8,S,s_{16},S,s_{24},\ldots,S,s_{2^d-8},S)$.
For example, we may take for~$R$ the well-known binary reflected Gray code in~$Q_d$.
Note that $\sigma(R)$ contains exactly $2^{d-3}$ copies of~$S$.
We define an ordering $x_1,\ldots,x_{2^d}$ of all elements of~$\langle C_k\rangle$ recursively by $x_1\ass \emptyset$ and $x_{i+1}\ass x_i\oplus c_{s_i}$ for $i=1,\ldots,2^d-1$.
Note that this definition is valid since each element of $\langle C_k\rangle$ is a linear combination of the basis~$C_k$, and $R$ visits all $d$-tuples of $\{0,1\}$-coefficients.

The symmetric difference of the edge sets of the cycle factor~$\cC$ with all 4-cycles~$F(x_i,c_{s_i})$ for $i=1,\ldots,2^d-1$ yields a Hamiltonian cycle in~$Q_n$, and removing one of its $n$-edges gives a Hamiltonian path~$P$ in~$Q_n$.
Note that the 4-cycles $F(x_i,c_{s_i})$ and $F(x_{i+1},c_{s_{i+1}})$ are edge-disjoint for $i=1,\ldots,2^d-2$ by Lemma~\ref{lem:Cdiff}.

In the remainder of the proof we compute the quantities~$\nu_{n-1}(P)$ and~$\lambda_{n-1}(P)$.
Recall from~\eqref{eq:Px} and~\eqref{eq:Cx} that for each cycle~$C(x_i)$, $i=1,\ldots,2^d$, the flip sequence is $\sigma(C(x_i))=(1,2,\ldots,n,1,2,\ldots,n)$, i.e., it has two maximal $(n-1)$-runs of length~$n-2$ each.
Observe from Figure~\ref{fig:run} that when gluing together the eight cycles~$C(x_i)$ for which $s_i\in\{1,2,3\}$ belongs to a copy of~$S$ in~$\sigma(R)$, then the flip sequence of the resulting cycle has the $(n-1)$-run partition
\begin{alignat*}{8}
   & \!\!\!\big(\mybox{3,2},\mybox{5},\mybox{4,\dots,n-1},\bcancel{n},\mybox{1,\dots,n-1},\bcancel{n}, \\
   & \mybox{3},\mybox{6},\mybox{3,\dots,n-1},\bcancel{n},\mybox{1,\dots,n-1},\bcancel{n}, \\
   & \mybox{3,2},\mybox{5},\mybox{4,\dots,n-1},\bcancel{n},\mybox{1,\dots,n-1},\bcancel{n}, \\
   & \mybox{3},\mybox{2,\dots,n-1},\bcancel{n},\mybox{1,\dots,n-1},\bcancel{n},\\
   & \mybox{3,2},\mybox{5},\mybox{4,\dots,n-1},\bcancel{n},\mybox{1,\dots,n-1},\bcancel{n}, \\
   & \mybox{3},\mybox{6},\mybox{3,\dots,n-1},\bcancel{n},\mybox{1,\dots,n-1},\bcancel{n},\\
   & \mybox{3,2},\mybox{5},\mybox{4,\dots,n-1},\bcancel{n},\mybox{1,\dots,n-1},\bcancel{n},\\
   & \mybox{3},\mybox{2,\dots,n-1},\bcancel{n},\mybox{1,\dots,n-1},\bcancel{n}\big),
\end{alignat*}
consisting of 30 runs of total length $16n-46$.

When taking the symmetric difference with the $2^{d-3}-1$ remaining 4-cycles $F(x_i,c_{s_i})$ with $s_i\geq 4$, then inserting each such 4-cycle splits two runs $\mybox{2,\dots,n-1}$ into two times two runs $\mybox{2,\dots,a-1}$, $\mybox{a+1,\dots,n-1}$ and two trivial runs $\mybox{b}$, where $\{a,b\}\ass c_{s_i}$ such that $a<b$, i.e., the number of runs increases by~4 and the total length decreases by~4.

Lastly, note that removing one $n$-edge from the resulting Hamiltonian cycle to break it into the Hamiltonian path~$P$ does not change any $(n-1)$-runs.

Combining these observations, we obtain
\begin{equation*}
\begin{split}
\nu_{n-1}(P) &= 30\cdot 2^{d-3}+4\cdot(2^{d-3}-1)=\frac{17}{8}\cdot\frac{2^n}{n}-4 \;\;\text{ and } \\
\lambda_{n-1}(P) &= (16n-46)\cdot2^{d-3}-4\cdot(2^{d-3}-1)=2^n-\frac{25}{8}\cdot \frac{2^n}{n}+4,
\end{split}
\end{equation*}
where we used the relation $2^d=\frac{2^n}{2n}$.
This completes the proof of the lemma.
\end{proof}

For a general dimension that is not necessarily a power of~2 we generalize the result from before by applying a straightforward product construction.

\begin{corollary}
\label{cor:run}
Let $k\geq 2$, $n\ass 2^k$, and $0\le m <n$.
There is a Hamiltonian path~$P$ in~$Q_{n+m}$ that satisfies
\[ \nu_{n-1}(P)=\begin{cases}
               6\cdot 2^m & \text{if } k=2, \\
               \big(\frac{17}{8}\cdot\frac{2^n}{n}-4\big)\cdot 2^m & \text{if } k\ge 3,
               \end{cases} \text{   and   }
\lambda_{n-1}(P)=\begin{cases}
                 8\cdot 2^m & \text{if } k=2, \\
                 \big(2^n-\frac{25}{8}\cdot\frac{2^n}{n}+4\big)\cdot 2^m & \text{if } k\ge 3.
                 \end{cases} \]
\end{corollary}

\begin{proof}
We use the fact that $Q_{n+m}$ is isomorphic to the Cartesian product $Q_n \mathbin{\square} Q_m$.
Consequently, if $P$ is a Hamiltonian path in~$Q_n$ and $R$ is a Hamiltonian path in~$Q_m$ with flip sequence $\sigma(R)\assr(s_1,\dots,s_{2^m-1})$, then $Q_{n+m}$ has a Hamiltonian path~$P'$ with flip sequence
\[ \sigma(P')=\big(\sigma(P),s_1+n,\rev(\sigma(P)),s_2+n,\sigma(P),\ldots,s_{2^m-2}+n,\sigma(P),s_{2^m-1}+n,\rev(\sigma(P))\big), \]
where $\rev(\sigma(P))$ denotes the reverse of the sequence~$\sigma(P)$.
By taking the Hamiltonian path~$P$ in~$Q_n$ given by Lemma~\ref{lem:run} and any Hamiltonian path $R$ in~$Q_m$, the lemma follows.
\end{proof}

Note that as $n\rightarrow \infty$, we have $\lambda_{n-1}(P)=(1-o(1))2^{n+m}$, i.e., almost all flips along the Hamiltonian path~$P$ belong to an $(n-1)$-run.

\section{Proof of Theorem~\ref{thm:min}}

In this section, we decribe our constructions of almost-minimum $n$-Venn diagrams for the case where $n$ is a power of 2, thus proving Theorem~\ref{thm:min}.

\begin{lemma}
\label{lem:min}
Let $k\geq 3$, $n\ass 2^k$, $d\ass 2^k-k-1$, and $\rho\ass 2^{k-1}-1=n/2-1$, and let $P$ be a Hamiltonian path in~$Q_d$.
Then there is an $n$-Venn diagram with exactly $2\cdot\frac{2^n}{n}-\nu_\rho(P)-2\lambda_\rho(P)-2$ many crossings.
\end{lemma}

\begin{proof}
To prove the lemma, we construct a plane subgraph~$H\seq Q_n$ satisfying properties~\circled{1}--\circled{3} stated in Section~\ref{sec:dual} that has exactly $2\cdot\frac{2^n}{n}-\nu_\rho(P)-2\lambda_\rho(P)-2$ many faces, as follows.
We start with the partition of~$Q_n$ into cycles given by Corollary~\ref{cor:iso-part-cycles}.
Formally, we define the cycle factor $\cC\ass \bigcup_{x\in\langle C_k\rangle}C(x)$.
Each of the cycles of~$\cC$ has length~$2n$ and the total number of cycles is~$2^n/(2n)=2^{2^k}/2^{k+1}=2^{2^k-k-1}=2^d$.
Recall that $|C_k|=2^k-k-1=d$.

Let $c=(c_1,\ldots,c_d)$ be a sequence obtained by sorting all elements of the set~$C_k$ such that all 2-sets from $O_k\subseteq C_k$ appear first and with increasing minimum values, and the remaining 2-sets appear afterwards in arbitrary order.
That is, we have $(c_1,\dots, c_\rho)=(\{1,3\},\{3,5\},\ldots,\{2^k-3,2^k-1\})$, which is the order from~\eqref{eq:OkCk}, or equivalently, $c_i=\{2i-1,2i+1\}$ for $1\le i \le \rho$.
Recall that $|O_k|=2^{k-1}-1=\rho\leq d$.

\begin{figure}
\centerline{
\includegraphics[page=5]{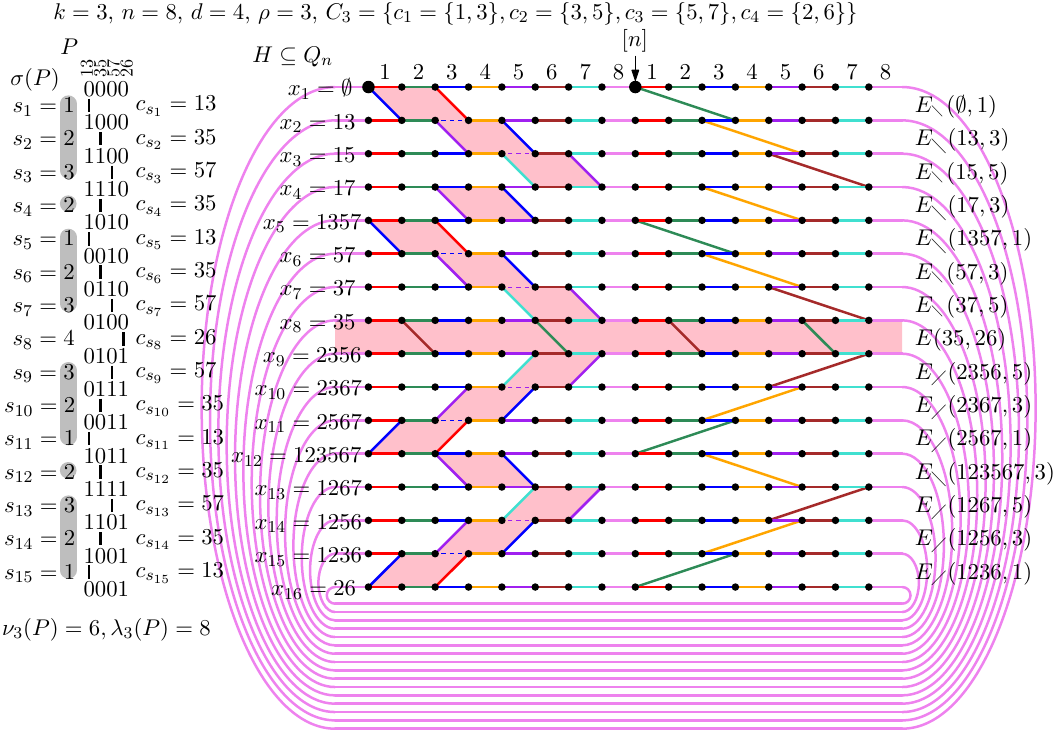}
}
\caption{Illustration of the proof of Lemma~\ref{lem:min} for the case~$k=3$, i.e., $n=8$.
The dashed edges are deleted, so they are not part of~$H$.
The constructed dual graph~$H$ of an 8-Venn diagram has 40 faces, i.e., the corresponding Venn diagram has 40 crossings; see Figure~\ref{fig:v8-primal}, only 3 more than the lower bound~$L_8=37$.
All faces except the red ones have maximum possible length~$2n$.}
\label{fig:v8-dual}
\end{figure}

We consider the flip sequence~$\sigma(P)\assr(s_1,\ldots,s_{2^d-1})$ of the given Hamiltonian path~$P$ in~$Q_d$.
The $\rho$-runs in this sequence will be particularly relevant for our construction.
We define an ordering $x_1,\ldots,x_{2^d}$ of all elements of~$\langle C_k\rangle$ recursively by $x_1\ass \emptyset$ and $x_{i+1}\ass x_i\oplus c_{s_i}$ for $i=1,\ldots,2^d-1$.
Note that this definition is valid since each element of~$\langle C_k\rangle$ is a linear combination of the basis~$C_k$, and $P$ visits all $d$-tuples of $\{0,1\}$-coefficients.

To construct~$H$, we first embed the cycles of the factor~$\cC$ in the plane by nesting them concentrically according to this ordering, i.e., $C(x_1)$ is the outermost cycle, and for $i=1,\ldots,2^d-1$, the cycle $C(x_{i+1})$ is nested concentrically inside~$C(x_i)$.

Next, we add edges between every pair of consecutive cycles~$C(x_i)$ and~$C(x_{i+1})$ for $i=1,\ldots,2^d-1$ according to the following rules; see Figure~\ref{fig:v8-dual}.
We first fix a $\rho$-run partition of~$P$ throughout, so whenever we refer to a $\rho$-run in the following, we mean a run in this fixed partition.
If $s_i\leq \rho$, then we have $c_{s_i}=\{a,a+2\}$ for $a\ass 2s_i-1$.
If the entry~$s_i$ of~$\sigma(P)$ is contained in an increasing $\rho$-run, then we add the three edges of~$E_{\diagdown}(x_i,a)$, and if it is contained in a decreasing $\rho$-run, then we add the three edges of~$E_{\diagup}(x_i,a)$ instead.
If the run has length~0, i.e., it consists only of a single entry, we treat it as one of the two cases, say, increasing.
On the other hand, if $s_i>\rho$, then we add the four edges of~$E(x_i,c_{s_i})$.

In this way, we obtain an intermediate connected plane graph with $C(x_1)$ as the outer face and $C(x_{2^d})$ as the innermost face that has either 3 or 4~faces between cycles~$C(x_i)$ and~$C(x_{i+1})$ for $i=1,\dots,2^d-1$, namely if $s_i \leq \rho$ or $s_i>\rho$ holds, respectively.
Since the number of entries~$s_i$ in $\sigma(P)$ with $s_i\leq \rho$ is exactly $\nu_\rho(P)+\lambda_\rho(P)$, the number of faces of the intermediate graph is exactly
\[ 2+4(2^d-1)-(\nu_\rho(P)+\lambda_\rho(P))=2\cdot \frac{2^n}{n}-\nu_\rho(P)-\lambda_\rho(P)-2. \]

In the last step of constructing~$H$, we remove edges from some of the cycles~$C(x_i)$, $i=2,\ldots,2^d-1$, according to the following rules:
If $s_{i-1}$ and~$s_i$ belong to the same increasing $\rho$-run in~$\sigma(P)$, then we have $c_{s_{i-1}}=\{a-2,a\}$ and~$c_{s_i}=\{a,a+2\}$ for $a\ass 2s_i-1$, and then we remove the first $a$-edge of~$C(x_i)$.
If $s_{i-1}$ and~$s_i$ belong to the same decreasing $\rho$-run in~$\sigma(P)$, then we have $c_{s_{i-1}}=\{a,a+2\}$ and~$c_{s_i}=\{a-2,a\}$ for $a\ass 2s_i+1$, and then we remove the first $a$-edge of~$C(x_i)$.
This completes the description of how to construct the subgraph~$H$ of~$Q_n$.

Note that each removed edge decreases the number of faces by~1.
Since the number of removed edges is exactly~$\lambda_\rho(P)$, the number of faces of~$H$ is exactly
\[ 2\cdot\frac{2^n}{n}-\nu_\rho(P)-2\lambda_\rho(P)-2, \]
as claimed.

It remains to check that the graph~$H\seq Q_n$ defined before satisfies conditions~\circled{1}--\circled{3}.

Condition~\circled{1} holds since the cycles $C(x_i)$ for $i=1,\dots,2^d$ form a partition of~$Q_n$ by Corollary~\ref{cor:iso-part-cycles}.

To verify condition~\circled{2}, first observe in Figure~\ref{fig:cross-edges}~(a1), (b1) and~(b2) that no two edges from any of the sets~$E(x,\{a,b\})$, $E_{\diagdown}(x,a)$, or $E_{\diagup}(x,a)$ are crossing and thus $H$ is a plane graph.
Furthermore, observe from Figure~\ref{fig:cross-edges}~(a1) that the four facial cycles between two cycles obtained by adding the edge set~$E(x,\{a,b\})$ have lengths $2(b-a)+2$ and $2(n-(b-a))+2$ and flip sequences
\begin{align*}
   & (a,a+1,\ldots,b-1,a,b,b-1,\dots,a+1,b), \text{ and} \\
   & (b,b+1,\ldots,n,1,2,\dots,a-1,b,a,a-1,\ldots,1,n,n-1,\dots,b+1,a),
\end{align*}
respectively.
Similarly, observe from Figures~\ref{fig:cross-edges}~(b1) and~(b2) that the three facial cycles between two cycles obtained by adding the edge set~$E_{\diagdown}(x,a)$ or~$E_{\diagup}(x,a)$ have lengths 6, $2n$, and $2n$ and flip sequences
\begin{align*}
   & (a,a+1,a,a+2,a+1,a+2), \\
   & (a+2,a+3,\ldots,n,1,2,\ldots,a-1,a+1,a+2,a+1,\dots,1,n,n-1,\ldots,a+3,a), \text{ and} \\
   & (a,a+1,\ldots,n,1,2,\ldots,a-1,a+2,a,a-1,\ldots,1,n,n-1,\dots,a+3,a+1),
\end{align*}
respectively.
It can be checked directly that the aforementioned faces, and also the outer face~$C(x_1)$ and the innermost face~$C(x_{2^d})$ satisfy condition~\circled{2}.

In the last step of the construction we removed some edges from the cycles~$C(x_i)$.
Specifically, for each increasing or decreasing run of length~$\ell\geq 2$, we glue together $\ell$ adjacent 6-cycles, creating a new facial cycle of length~$4\ell+2$.
Observe that the adjacent 6-faces have edge directions shifted by~2 and therefore their directions overlap only in the largest resp.\ smallest direction, and the shared direction is the direction of the edge that was removed.
It follows that the resulting cycles of length~$4\ell+2$ also satisfy condition~\circled{2}.

To prove condition~\circled{3} we establish by induction on $i=1,\dots,2^d$ that for every direction $j\in [n]$ every vertex~$x$ of~$C(x_i)$ is connected in~$H$ to either the vertex~$\emptyset$ or~$[n]$ along a path that avoids any $j$-edges.
Specifically, $x$ is connected to~$\emptyset$ if $j\notin x$ and to~$[n]$ if $j\in x$.

To settle the induction basis $i=1$, recall that the flip sequence of~$C(x_1)$ is $\sigma(C(x_1))=(1,2,\ldots,n,1,2,\ldots,n)$, so removing any two $j$-edges for some $j\in[n]$ from~$C(x_1)$ results in two antipodal paths of length~$n-1$, one containing the vertex~$\emptyset$ and the other containing the vertex~$[n]$.

For the induction step~$i\rightarrow i+1$, it suffices to show that every vertex of~$C(x_{i+1})$ is connected in~$H$ to some vertex of~$C(x_i)$ along a path that has no $j$-edges.
We distinguish two cases according to which edges were added between~$C(x_i)$ and~$C(x_{i+1})$.

In the case when the four edges of $E(x_i,c_{s_i})$, $c_{s_i}\assr\{a,b\}$, were added between~$C(x_i)$ and~$C(x_{i+1})$, observe that the two added $a$-edges are incident to antipodal vertices of~$C(x_{i+1})$, and the same is true for the two added $b$-edges.
Hence removing any two $j$-edges, where $j\in[n]$, from $C(x_{i+1})$ results in two subpaths that are both connected to~$C(x_i)$ along an edge of~$E(x_i,c_{s_i})$ that is not a $j$-edge, even if $j\in \{a,b\}$.

\begin{figure}[h!]
\centerline{
\includegraphics[page=3]{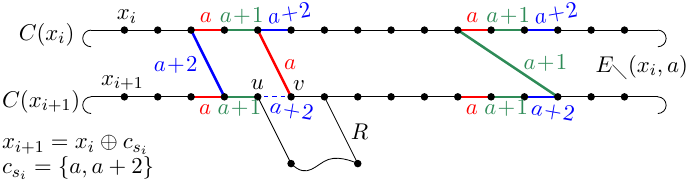}
}
\caption{Verification of condition~\protect\circled{3} in the proof of Lemma~\ref{lem:min}.}
\label{fig:cond3}
\end{figure}

Now we consider the case when the three edges of~$E_{\diagdown}(x_i,a)$ were added between~$C(x_i)$ and $C(x_{i+1})$ (i.e., we have $b=a+2$).
The third case when the edges of~$E_{\diagup}(x_i,a)$ were added is analogous.
We will also assume w.l.o.g.\ that the first $(a+2)$-edge on~$C(x_{i+1})$ was removed; otherwise the situation is even simpler; see Figure~\ref{fig:cond3}.
We denote the removed $(a+2)$-edge by~$\{u,v\}$.
If $j\notin\{a,a+1,a+2\}$ then removing the $j$-edges from $C(x_{i+1})$ results in three subpaths that are all connected to~$C(x_i)$, each path along an edge of~$E_{\diagdown}(x_i,a)$.
If $j=a+2$ then removing also the second $(a+2)$-edge from~$C(x_{i+1})$ results in two subpaths of~$C(x_{i+1})$ that are both connected to~$C(x_i)$, since the added edges of directions~$a$ and~$a+1$ are incident to antipodal vertices of~$C(x_{i+1})$.
For the last case $j\in \{a,a+1\}$ recall that there is a path~$R$ between~$u$ and~$v$ along the face of~$H$ containing these two vertices and vertices of~$C(x_{i+j})$ for $j\geq 2$, and the path~$R$ does not contain edges of directions~$a$ or~$a+1$.
The path~$R$ and the two edges of $E_{\diagdown}(x_i,a)$ that are not of direction~$j$ connect all vertices of~$C(x_{i+1})$ to $C(x_i)$ along a path not containing any $j$-edges.
This shows that the graph~$H$ indeed satisfies condition~\circled{3}.

The proof of the lemma is complete.
\end{proof}

With Lemma~\ref{lem:min} in hand, we are now in position to prove Theorem~\ref{thm:min}.

\begin{proof}[Proof of Theorem~\ref{thm:min}]
For $k\ge 3$ we define $n\ass 2^k$, $d\ass 2^k-k-1$ and $\rho\ass 2^{k-1}-1=n/2-1$.

If $k=3$, we have $n=8$, $d=4$ and~$\rho=3$.
We apply Lemma~\ref{lem:run} for $\wh{k}\ass k-1=2$, $\wh{n}\ass 2^{\wh{k}}=4$ to obtain a Hamiltonian path~$P$ in~$Q_4$ with $\nu_3(P)=6$ and $\lambda_3(P)=8$.
Therefore, Lemma~\ref{lem:min} yields an $8$-Venn diagram with exactly
\[ 2\cdot\frac{2^8}{8}-6-2\cdot 8-2=40 \]
crossings; see Figures~\ref{fig:v8-primal} and~\ref{fig:v8-dual}.

If $k\ge 4$, we apply Corollary~\ref{cor:run} for $\wh{k}\ass k-1$, $\wh{n}\ass 2^{\wh{k}}=2^{k-1}=n/2$ and $\wh{m}\ass d-\wh{n}=2^{k-1}-k-1$ to obtain a Hamiltonian path~$P$ in~$Q_{\wh{n}+\wh{m}}=Q_d$ with
\[ \nu_\rho(P)=\nu_{\wh{n}-1}(P)=\left(\frac{17}{8}\cdot\frac{2^{\wh{n}}}{\wh{n}}-4\right)2^{\wh{m}}=\frac{17}{8}\cdot\frac{2^n}{n^2}-\frac{2^{n/2}}{n/2} \]
and
\[ \lambda_\rho(P)=\lambda_{\wh{n}-1}(P)=\left(2^{\wh{n}}-\frac{25}{8}\cdot \frac{2^{\wh{n}}}{\wh{n}}+4\right)2^{\wh{m}}=\frac{2^n}{2n}-\frac{25}{8}\cdot\frac{2^n}{n^2}+\frac{2^{n/2}}{n/2}. \]
Therefore, Lemma~\ref{lem:min} yields an $n$-Venn diagram with exactly
\[ 2\cdot\frac{2^n}{n}-\nu_\rho(P)-2\lambda_\rho(P)-2=\left(1+\frac{33}{8n}-\frac{2}{2^{n/2}}-\frac{2n}{2^n}\right)\frac{2^n}{n} \]
many crossings.

Finally, note that the inequality~$\frac{2^n}{n}\le L_n=\lceil\frac{2^n-2}{n-1}\rceil$ holds for all~$n\ge 2$.
\end{proof}

\section{Proof of Theorem~\ref{thm:min-all}}

We now show how to lift the construction of $n$-Venn diagrams for the case when $n$ is a power of~2 to the general case, thus proving Theorem~\ref{thm:min-all}.
The lifting is achieved via the following straightforward doubling construction.

We say that a face of the dual graph of an $n$-Venn diagram is \defi{colorful} if it has length~$2n$ and contains two antipodal vertices of~$Q_n$.
In the primal Venn diagram, such a face corresponds to a crossing involving all $n$ curves for which the cyclic ordering of curves around the crossing can be split into two halves such that every curve appears exactly once in each half.
Consequently, there is a way to draw an additional curve through this crossing that crosses (not only intersects) each of the $n$ existing curves.

\begin{lemma}
\label{lem:doubling}
If there is an $n$-Venn diagram $D$ whose dual graph has a colorful face, then there is an $(n+1)$-Venn diagram with twice as many crossings as~$D$ whose dual graph has a colorful face.
\end{lemma}

\begin{proof}
For a set or sequence~$X$ of subsets of~$[n]$ we write $X\cup\{n+1\}$ for the set or sequence of subsets obtained by adding the element~$n+1$ to each subset.
The operation $G\cup\{n+1\}$ for a graph~$G$ with vertex set $V(G)\subseteq 2^{[n]}$ is defined similarly.

\begin{figure}[h!]
\centerline{
\includegraphics{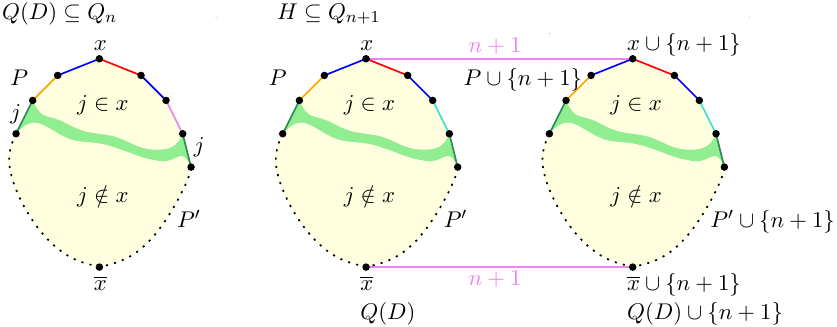}
}
\caption{Illustration of the proof of Lemma~\ref{lem:doubling}.}
\label{fig:doubling}
\end{figure}

The proof of the lemma uses the dual graph~$Q(D)$; recall Section~\ref{sec:dual}.
We may assume w.l.o.g.\ that the colorful face in~$Q(D)$ is the outer face; see Figure~\ref{fig:doubling}.
Let $x$ and $\ol{x}$ be two antipodal vertices on the outer face, and let~$P$ and~$P'$ be the two paths between $x$ and $\ol{x}$ along the outer face.
The flip sequences~$\sigma(P)$ and~$\sigma(P')$ are both permutations of~$[n]$, i.e., each direction of~$[n]$ is flipped exactly once, but possibly in a different order.
We construct the dual graph~$H$ of an $(n+1)$-Venn diagram as follows:
We take two copies of~$Q(D)$, namely $Q(D)$ and $Q(D)\cup\{n+1\}$, embedded into the plane next to each other, and we add the two edges $\{x,x\cup\{n+1\}\}, \{\ol{x},\ol{x}\cup\{n+1\}\}$ between them.
We claim that the resulting plane graph~$H\seq Q_{n+1}$ satisfies conditions \circled{1}-\circled{3}.
Condition~\circled{1} is clear.
Since $Q(D)$ satisfies condition~\circled{2} and the two facial cycles of~$H$ of length~$2(n+1)$ have the flip sequences $(\sigma(P),n+1,\sigma(P'),n+1)$ and $(\sigma(P'),n+1,\sigma(P),n+1)$, the condition~\circled{2} also holds in~$H$.
To verify condition~\circled{3}, let $j\in[n+1]$, and consider the subgraphs of~$H$ induced by all vertices~$x$ with $j\in x$, and with $j\notin x$, respectively.
We distinguish two cases.
If $j\le n$ note that $x$ and $\ol{x}$ are in different induced subgraphs of~$Q(D)$, one for the vertices~$x$ with $j\in x$, the other with $j\notin x$.
Consequently, the edges $\{x,x\cup\{n+1\}\}$ and $\{\ol{x},\ol{x}\cup\{n+1\}\}$ connect the respective induced subgraphs of~$Q(D)$ and~$Q(D)\cup\{n+1\}$ with each other.
On the other hand, if $j=n+1$ then the two induced subgraphs of~$H$ are $Q(D)$ and $Q(D)\cup\{n+1\}$, which are indeed connected.
We conclude that~$H$ satisfies conditions~\circled{1}--\circled{3}, so it is the dual graph of an $(n+1)$-Venn diagram.

Note that the outer face of~$H$ is colorful, as it contains the two antipodal vertices~$x$ and $\ol{x}\cup\{n+1\}$.
Finally, it is easy to see that $H$ has twice as many faces as~$Q(D)$, and therefore the corresponding Venn diagram has twice as many crossings as~$D$.
\end{proof}

With Lemma~\ref{lem:doubling} in hand, we are now in position to prove Theorem~\ref{thm:min-all}.

\begin{proof}[Proof of Theorem~\ref{thm:min-all}]
The duals of the Venn diagrams constructed for Theorem~\ref{thm:min} all have a colorful face, namely the outer face~$C(x_1)$, so applying Lemma~\ref{lem:doubling} ($m$ times) proves the theorem.
\end{proof}

\section{Remarks and open problems}
\label{sec:open}

We conclude with some remarks and challenging open problems.

\begin{itemize}[leftmargin=4mm]
\item
Is there a minimum 8-Venn diagram, i.e., one with only 37 crossings?
The best one we found has 40 crossings; see Figure~\ref{fig:v8-primal}.

\item
How many (not necessarily simple) $n$-Venn diagrams are there for $n=3,4,5$?
How many of them are minimum?

\item
Lemma~\ref{lem:run} gives a Hamiltonian path~$P$ in~$Q_n$, $n=2^k$, with
\[ \nu_{n-1}(P)+2\lambda_{n-1}(P)=2\cdot 2^n-\frac{33}{8}\cdot\frac{2^n}{n}+4. \]
Is there another construction that would decrease the constant~$\frac{33}{8}$?
If so, then this would directly improve the same constant in Theorems~\ref{thm:min} and~\ref{thm:min-all}.

Since every $(n-1)$-run contains at most $(n-1)$ elements (length at most $n-2$) and is separated from the next such run by at most one flip~$n$, we have $\nu\geq 2^n/n$ (for simplicity, we ignore the subscript~$n-1$ and the argument~$P$).
Using that $\nu+\lambda\le 2^n$ yields the upper bound
\[ \nu+2\lambda=2(\nu+\lambda)-\nu \le 2\cdot 2^n-\frac{2^n}{n}. \]

\item
A long-standing problem due to Slater~\cite{slater_1979} is whether there exists a Hamiltonian path~$P$ in~$Q_n$ such that any two consecutive entries of the flip sequence differ by~$\pm 1$.
Put differently, if we write $\mu(P)$ for the number of such consecutive entries of~$\sigma(P)$, then the goal is to find a Hamiltonian path~$P$ with $\mu(P)=2^n-2$.
While such a path exists for $n\le 6$, it does not exist for $n=7$~\cite{slater_1989} and $n=8$ (computer experiments by Dimitrov, Gregor and Lu\v{z}ar; see \cite[Problem~10]{MR4649606}).
We clearly have $\mu(P)\geq\lambda_n(P)\geq \lambda_{n-1}(P)$, so Lemma~\ref{lem:run} and Corollary~\ref{cor:run} yield a Hamiltonian path~$P$ for which $\mu(P)=(1-o(1))\cdot (2^n-2)$, which can be seen as an approximate solution to Slater's problem.
\end{itemize}

\bibliographystyle{alpha}
\bibliography{refs}

\end{document}